\documentclass[a4paper,10pt]{amsart}
\usepackage[shortlabels]{enumitem}
\usepackage{amsthm,amsmath,amssymb,graphics,graphicx,hyperref,epstopdf,mathrsfs}
\usepackage{mathtools}
\usepackage{breqn}      
\usepackage{tikz}     
\usepackage{capt-of}
\title{Betti numbers of monomial ideals in four variables}
\author{Guillermo Alesandroni}
\address{2000 Rosario, Santa Fe, Argentina}
\email{guillea@okstate.edu, alesandronig@yahoo.com}

\newtheorem{theorem}{Theorem}[section]

\newtheorem{corollary}[theorem]{Corollary}
\newtheorem{lemma}[theorem]{Lemma}

\theoremstyle{definition}
\newtheorem{definition}[theorem]{Definition}

\newtheorem{example}[theorem]{Example}
\newtheorem{construction}[theorem]{Construction}

\newtheorem{note}[theorem]{Note}

\DeclareMathOperator{\betti}{b}
\DeclareMathOperator{\pd}{pd}

\DeclareMathOperator{\lcm}{lcm}

\DeclareMathOperator{\Max}{max}

\begin{document}
\maketitle
\begin{abstract}
We express the multigraded Betti numbers of monomial ideals in 4 variables in terms of the multigraded Betti numbers of 66 squarefree monomial ideals, also in 4 variables. We use this class of 66 ideals to prove that monomial resolutions in 4 variables are independent of the base field. In addition, we give a formula for the Betti numbers of an arbitrary monomial ideal in 4 variables.
\end{abstract}

\section{Introduction}
Motivated by the work of Ezra Miller [Mi], this article is entirely concerned with monomial resolutions in 4 variables. The key argument in this study is the fact that the multigraded Betti numbers of an arbitrary monomial ideal in 4 variables can be expressed in terms of the multigraded Betti numbers of squarefree ideals, also in 4 variables. 

Every monomial ideal can be expressed as a squarefree ideal by polarizing. However, the technique of polarization requires adjoining many new variables. Thus, the polarization of an ideal in 4 variables will usually be an ideal in more than 4 variables. The idea that we introduce in this paper is different in the sense that we always work in 4 variables. Indeed, we reduce the general case to the study of 66 squarefree tetravariate ideals. After considering each of these 66 cases, we conclude that monomial resolutions in 4 variables are independent of the base field, and we construct a formula for the Betti numbers of an arbitrary monomial ideal in 4 variables. 

Our work is organized as follows. Section 2 concerns background and notation. In Section 3, we explain how to express the multigraded Betti numbers of an ideal in terms of the multigraded Betti numbers of 66 squarefree ideals. Section 4 gives the list of 66 squarefree ideals mentioned above. In section 5, we prove that resolutions in 4 variables are characteristic-independent. Section 6 gives a formula for the Betti numbers of all ideals in 4 variables. Section 7 discusses the advantage of formulas over algorithms. In Section 8, we close the article with questions and final thoughts.

 \section{Background and notation}
 Throughout this paper $S$ represents a polynomial ring over an arbitrary field $k$, in 4 variables. The letter $M$ always denotes a monomial ideal
in $S$, and the symbol $\mathbb{T}_M$ always represents the Taylor resolution of $S/M$. If $m$ is the multidegree of a basis element of $\mathbb{T}_M$, sometimes we will say that $m$ is a multidegree of $\mathbb{T}_M$, for short. The unconventional notation $m \in \mathbb{T}_M$ will also convey this idea. A nice construction of the Taylor resolution as a multigraded free resolution can be found in [Me].

\begin{definition}
Let $M$ be minimally generated by a set of monomials $G$.
\begin{itemize}
\item A monomial $m\in G$ is called \textbf{dominant} (in $G$) if there is a variable $x$, such that for all $m'\in G\setminus\{m\}$, the exponent with 
which $x$ appears in the factorization of $m$ is larger than the exponent with which $x$ appears in the factorization of $m'$.
We say that $G$ is a \textbf{dominant set} if each of its elements is dominant. The ideal $M$ is a \textbf{dominant ideal} if $G$ is a dominant set.
\item $G$ is called \textbf{$p$-semidominant} if $G$ contains
exactly $p$ nondominant monomials. The ideal $M$ is \textbf{$p$-semidominant} if $G$ is $p$-semidominant.
\end{itemize}
\end{definition}

 \begin{example}\label{example 1}
  Let $M_1$, $M_2$, and $M_3$ be minimally generated by $G_1=\{a^2,b^3,ab\}$, $G_2=\{ab,bc,ac\}$, and $G_3=\{a^2b,ab^3c,bc^2,ad^2\}$, respectively. Note that $a^2$ and $b^3$ are dominant in $G_1$, but $ab$ is not. Thus, both the set $G_1$ and the ideal $M_1$ are {1}-semidominant. On the other hand, $ab$, $bc$, and $ac$ are nondominant in $G_2$. Therefore, $G_2$ and $M_2$ are {3}-semidominant. Finally, $a^2b$, $ab^3c$, $bc^2$, and $ad^2$ are dominant in $G_3$. Thus, $G_3$ and $M_3$ are dominant.
\end{example}

The next theorem gives a complete characterization of when the Taylor resolution is minimal [Al].

\begin{theorem} \label{Taylor}
$\mathbb{T}_M$ is minimal if and only if $M$ is dominant.
\end{theorem}

The formula for the fourth Betti numbers of $S/M$ is already known [Al2, Corollary 6.3 (i)], and we will state it below. First, we will need a few definitions.

 Let $m_1=x_1^{\alpha_1}\ldots x_4^{\alpha_4}$, and $m_2=x_1^{\beta_1}\ldots x_4^{\beta_4}$ be two monomials of $S$. We say that  $m_1$ \textbf{strongly divides} $m_2$, 
 if $\alpha_i<\beta_i$, whenever $\alpha_i\neq 0$. For instance, $m_1= x_1 x_2$ strongly divides $m_2= x_1^2 x_2^2 x_3 x_4$; but $m_3= x_1 x_2 x_3$ does not strongly divide $m_2$, as $x_3$ appears with exponent 1 in the factorizations of $m_2$ and $m_3$.

Let $G$ be the minimal generating set of $M$. Define the class
 $\mathscr{D}_M=
 \{D \subseteq G: D$ is a dominant set of cardinality $4$, such that no element of $G$ strongly divides $\lcm D$$\}$.

\begin{theorem} \label{fourthbetti}
Let $L=\{m: m=\lcm D \text{, for some } D\in \mathscr{D}_M\}$. Then, $\betti_4(S/M)=\# L$.
\end{theorem}

For example, if $M$ has minimal generating set $G=\{x_1^2, x_2^2, x_3^2, x_1 x_4^2, x_2 x_4^2\}$, then the class $\mathscr{D}_M$ consists of only two sets: $D_1=\{x_1^2, x_2^2, x_3^2, x_1 x_4^2\}$, and $D_2=\{x_1^2, x_2^2, x_3^2, x_2 x_4^2\}$. Since $\lcm D_1=\lcm D_2=x_1^2 x_2^2 x_3^2 x_4^2$, it follows from Theorem \ref{fourthbetti} that $\betti_4(S/M)=1$. 

\section{Reducing to the square free case}

In this section we explain how to interpret the Betti numbers of an ideal in 4 variables in terms of the Betti numbers of squarefree ideals in 4 variables. Part of this material has been taken from [Al2] and adapted to the study of free resolutions in 4 variables. The next theorem is due to Gasharov, Hibi, and Peeva [GHP].

\begin{theorem}\label{Theorem 1}
Let $M$ be minimally generated by $G$, and consider a multidegree $m$ of $\mathbb{T}_M$. Let $M_m$ be the ideal generated by all monomials of $G$ dividing $m$. Then $\betti_{i,m}(S/M) = \betti_{i,m}(S/M_m)$, for all $i$.
\end{theorem}

\begin{construction}\label{Construction 2}
Let $M$ be minimally generated by $G$, and consider a multidegree $m$ of $\mathbb{T}_M$. Let $M_m = (m_1,\ldots,m_q)$ be the ideal minimally generated by all monomials of $G$ dividing $m$. \\
If $m_i =x_1^{\alpha_{i1}}.\ldots.x_4^{\alpha_{i4}}$, with $1\leq i\leq q$, then 
$m= x_1^{\alpha_1}.\ldots.x_4^{\alpha_4}$, with $\alpha_j = \Max(\alpha_{1j},\ldots,\alpha_{qj})$. \\
For each $i=1,\ldots,q$, define 
\[m'_i = x_1^{\beta_{i1}}.\ldots.x_4^{\beta_{i4}} \text{, where }\beta_{ij} = \begin{cases}
														\alpha_j &\text{ if } \alpha_{ij} = \alpha_j,\\
															0 &\text{ otherwise}.
														\end{cases}\]
Let $M'_m = (m'_1,\ldots,m'_q)$. The ideal $M'_m$ will be referred to as the \textbf{twin ideal} of $M_m$. For future reference, the minimal generating sets of  $M_m$ and $M'_m$ will be denoted by $G_m$ and $G'_m$, respectively.
\end{construction}

\begin{example}\label{Example 3}
Let $M = (x_1^3,x_1^2x_2^2,x_3^2x_4^2,x_1^2 x_2 x_3,x_2 x_3 x_4^2)$. Consider the multidegree $m = x_1^3 x_2^2 x_3 x_4^2$ of $\mathbb{T}_M$. Then $M_m = (x_1^3,x_1^2 x_2^2,x_1^2x_2x_3,x_2 x_3 x_4^2)$, and $M'_m = (x_1^3,x_2^2,x_3,x_3 x_4^2) = (x_1^3,x_2^2,x_3)$. This example shows that even if $\{m_1,\ldots,m_q\}$ is the minimal generating set of $M_m$, $\{m'_1,\ldots,m'_q\}$ may not be the minimal generating set of $M'_m$.
\end{example}

\begin{construction}\label{Construction 4}
We continue to use the notation of Construction \ref{Construction 2}. Let $y_1 = x_1^{\alpha_1},\ldots,y_4 = x_4^{\alpha_4}$, and denote by $\{ \alpha_{j_1}, \ldots ,\alpha_{j_k} \}$ the set of all nonzero exponents $\alpha_i$.  Define $y_m = y_{j_1} \ldots  y_{j_k}$,  and $T = k[y_1,\ldots,y_4]$.\\
For each $i = 1,\ldots,q$, let
\[m''_i = y_1^{\delta_{i1}}.\ldots.y_4^{\delta_{i4}}, \text{ where }\delta_{ij} = \begin{cases}
															1 & \text{if } \beta_{ij} = \alpha_j,\\
															0 & \text{otherwise}.
															\end{cases}\]
The squarefree ideal $M''_m = (m''_1,\ldots,m''_q)$ will be called the \textbf{squarefree twin ideal} of $M_m$. The minimal generating set of $M''_m$ will be denoted by $G''_m$. Note that the difference between $M'_m$ and $M''_m$ is only psychological, as $m''_i$ is just another representation of $m'_i$. Also, since $y_m = y_{j_1}\ldots y_{j_k} = x_{j_1}^{\alpha_{j_1}}\ldots x_{j_k}^{\alpha_{j_k}} = m$, it follows that $\betti_{i,m}(S/M'_m) = \betti_{i,y_m}(T/M''_m)$.
\end{construction}

\begin{example}\label{Example 5}
In Example \ref{Example 3}, $m = x_1^3 x_2^2 x_3 x_4^2$, and $M'_m = (x_1^3,x_2^2,x_3)$. According to Construction \ref{Construction 4}, $y_1 = x_1^3$, $y_2 = x_2^2$, $y_3 = x_3$, $y_4 = x_4^2$; $y_m=y_1 y_2 y_3 y_4$; and $M''_m = (y_1,y_2,y_3)$.
\end{example}

\begin{note}\label{Note 6} $M''_m$ is well defined because each variable $x_i$ appears with the same nonzero exponent in the factorization of all generators of $M'_m$ that are divisible by $x_i$. This is an important property of twin ideals without which the concept of squarefree twin ideal would not make sense. For instance, if there were a twin ideal of the form $M'_m=(x_1x_2,x_2^2x_3^2,x_3x_4^3)$, then we would have $y_1=x_1$, $y_2=x_2^2$, $y_3=x_3^2$, and $y_4=x_4^3$; and it would not be possible to represent the generator $x_1x_2$ in terms of the variables $y_i$.
\end{note}

 The next theorem, whose proof relies on a sequence of technical lemmas, can be found in [Al2, Theorem 4.10].

\begin{theorem}\label{Theorem 7}
Let $m$ be a multidegree of $\mathbb{T}_M$. Then $\betti_{i,m}(S/M_m) = \betti_{i,m}(S/M'_m)$, for all $i$.
\end{theorem}

\begin{corollary}\label{Corollary 8}
Let $m$ be a multidegree of $\mathbb{T}_M$. Then $\betti_{i,m}(S/M) = \betti_{i,y_m}(T/M''_m)$, for all $i$.
\end{corollary}

\begin{proof}
By Theorems \ref{Theorem 1} and \ref{Theorem 7}, 
\[\betti_{i,m}(S/M) = \betti_{i,m}(S/M_m) = \betti_{i,m}(S/M'_m) = \betti_{i,y_m}(T/M''_m).\]
\end{proof} 

\begin{corollary}\label{Corollary 9}
$\betti_i(S/M) = \sum\limits_{m\in \mathbb{T}_M} \betti_{i,y_m}(T/M''_m)$.
\end{corollary}

\begin{proof}
By Corollary \ref{Corollary 8}, 
\[\betti_i(S/M) = \sum\limits_{m\in \mathbb{T}_M} \betti_{i,m}(S/M) = \sum\limits_{m\in \mathbb{T}_M} \betti_{i,y_m}(T/M''_m).\]
\end{proof}

\section{Squarefree ideals in 4 variables}\label{Section 2}

Corollary \ref{Corollary 9} says that the Betti numbers of $M$ can be expressed in terms of the multigraded Betti numbers of the squarefree twin ideals $M''_m$, with $m\in \mathbb{T}_M$. Since the number of squarefree ideals in 4 variables is finite (and small enough to list them one by one), the study of Betti numbers of an infinite family can be reduced to the study of the Betti numbers of a handful of them. With this purpose, we now describe all squarefree ideals in 4 variables.

We will express the class of all squarefree ideals $\mathscr{M}$ as the disjoint union $\bigcup\limits_{i=0}^4 \mathscr{M}_i$, where $\mathscr{M}_i$ is the class of all squarefree ideals for which the largest degree of a minimal generator is $i$. Below, we describe the elements of each $\mathscr{M}_i$. \\
0) $\mathscr{M}_0: \overbracket{(1)}^{\#\textbf{1}} = S$.\\
1) Let $M$ be an arbitrary ideal in $\mathscr{M}_1$. Then $M = (x_{i_1},\ldots,x_{i_r})$, with $1\leq i_1<\ldots<i_r\leq 4$. After making the change of variables $y_1=x_{i_1},\ldots,y_r = x_{i_r}$, $M$ can be expressed as one of the following ideals:\\
$\overbracket{(y_1)}^{\#\textbf{2}}$ $ \overbracket{(y_1,y_2)}^{\#\textbf{3}}$  $\overbracket{(y_1,y_2,y_3)}^{\#\textbf{4}}$ $ \overbracket{(y_1,y_2,y_3,y_4)}^{\#\textbf{5}}$.\\
2) Let $M$ be an arbitrary ideal in $\mathscr{M}_2$. Then, $M$ is of the form
\[(x_{i_1} x_{j_1},\ldots,x_{i_r} x_{j_r},x_{k_1},\ldots,x_{k_s}).\]
Let $ij$ be the first element among $i_1j_1,\ldots,i_rj_r$ in lexicographic order. Let $y_1 = x_i$, and $y_2 = x_j$. Let $\{x_k,x_l\} = \{x_1,\ldots,x_4\}\setminus \{x_i,x_j\}$, where $k<l$. If $x_k$ does not appear in the factorization of any of the minimal generators of $M$, define $y_3 = x_l$. Otherwise, let $y_3 = x_k$ and $y_4 = x_l$. Then, $M$ is one of the following ideals:
\begin{itemize}[leftmargin=*]
\item Ideals with exactly one minimal generator: \[\overbracket{(y_1 y_2)}^{\#\textbf{6}}\]
\item Ideals with exactly two minimal generators: 
\[\overbracket{(y_1 y_2,y_1y_3)}^{\#\textbf{7}}\: \overbracket{(y_1y_2,y_1y_4)}^{\#\textbf{8}}\: \overbracket{(y_1y_2,y_2y_3)}^{\#\textbf{9}}\: \overbracket{(y_1y_2,y_2y_4)}^{\#\textbf{10}}\: \overbracket{(y_1y_2,y_3y_4)}^{\#\textbf{11}} \:
\overbracket{(y_1y_2,y_3)}^{\#\textbf{12}}\: \overbracket{(y_1y_2,y_4)}^{\#\textbf{13}}\]
\item Ideals with exactly three minimal generators: 
\[\overbracket{(y_1y_2,y_1y_3,y_1y_4)}^{\#\textbf{14}}\quad \overbracket{(y_1y_2,y_1y_3,y_2y_3)}^{\#\textbf{15}} \quad \overbracket{(y_1y_2,y_1y_3,y_2y_4)}^{\#\textbf{16}} \quad \overbracket{(y_1y_2,y_1y_3,y_3y_4)}^{\#\textbf{17}}\]
\[\overbracket{(y_1y_2,y_1y_4,y_2y_3)}^{\#\textbf{18}} \quad \overbracket{(y_1y_2,y_1y_4,y_2y_4)}^{\#\textbf{19}} \quad \overbracket{(y_1y_2,y_1y_4,y_3y_4)}^{\#\textbf{20}} \quad \overbracket{(y_1y_2,y_2y_3,y_2y_4)}^{\#\textbf{21}}\]
\[ \overbracket{(y_1y_2,y_2y_3,y_3y_4)}^{\#\textbf{22}}\quad \overbracket{(y_1y_2,y_2y_4,y_3y_4)}^{\#\textbf{23}} \quad \overbracket{(y_1y_2,y_1y_3,y_4)}^{\#\textbf{24}} \quad \overbracket{(y_1y_2,y_2y_3,y_4)}^{\#\textbf{25}} \]
\[ \overbracket{(y_1y_2,y_1y_4,y_3)}^{\#\textbf{26}} \quad \overbracket{(y_1y_2,y_2y_4,y_3)}^{\#\textbf{27}} \quad \overbracket{(y_1y_2,y_3,y_4)}^{\#\textbf{28}}\]
\item Ideals with exactly four minimal generators:
\[\overbracket{(y_1y_2,y_1y_3,y_1y_4,y_2y_3)}^{\#\textbf{29}} \quad \overbracket{(y_1y_2,y_1y_3,y_1y_4,y_2y_4)}^{\#\textbf{30}} \quad \overbracket{(y_1y_2,y_1y_3,y_1y_4,y_3y_4)}^{\#\textbf{31}},\]
\[\overbracket{(y_1y_2,y_1y_3,y_2y_3,y_2y_4)}^{\#\textbf{32}}\quad
 \overbracket{(y_1y_2,y_1y_3,y_2y_3,y_3y_4)}^{\#\textbf{33}} \quad \overbracket{(y_1y_2,y_1y_3,y_2y_4,y_3y_4)}^{\#\textbf{34}}\]
\[\overbracket{(y_1y_2,y_1y_4,y_2y_3,y_2y_4)}^{\#\textbf{35}} \quad \overbracket{(y_1y_2,y_1y_4,y_2y_3,y_3y_4)}^{\#\textbf{36}} \quad \overbracket{(y_1y_2,y_1y_4,y_2y_4,y_3y_4)}^{\#\textbf{37}} \]
\[\overbracket{(y_1y_2,y_2y_3,y_2y_4,y_3y_4)}^{\#\textbf{38}} \quad \overbracket{(y_1y_2,y_1y_3,y_2y_3,y_4)}^{\#\textbf{39}} \quad \overbracket{(y_1y_2,y_1y_4,y_2y_4,y_3)}^{\#\textbf{40}}\]
\item Ideals with exactly five minimal generators:\\
\[\overbracket{(y_1y_2,y_1y_3,y_1y_4,y_2y_3,y_2y_4)}^{\#\textbf{41}}\quad \overbracket{(y_1y_2,y_1y_3,y_1y_4,y_2y_3,y_3y_4)}^{\#\textbf{42}}\quad \overbracket{(y_1y_2,y_1y_3,y_1y_4,y_2y_4,y_3y_4)}^{\#\textbf{43}}\]
\[\overbracket{(y_1y_2,y_1y_3,y_2y_3,y_2y_4,y_3y_4)}^{\#\textbf{44}}\quad \overbracket{(y_1y_2,y_1y_4,y_2y_3,y_2y_4,y_3y_4)}^{\#\textbf{45}}\]
\item Ideals with exactly six minimal generators:
\[\overbracket{(y_1y_2,y_1y_3,y_1y_4,y_2y_3,y_2y_4,y_3y_4)}^{\#\textbf{46}}\]
3) Let $M$ be an arbitrary ideal in $\mathscr{M}_3$. Then $M$ is of the form $M = (x_{i_1}x_{j_1}x_{k_1},\ldots,x_{i_r}x_{j_r}x_{k_r},$ monomials of degree less than $3)$.\\
Let $ijk$ be the first element among $i_1j_1k_1,\ldots,x_rj_rk_r$ in lexicographic order. Let $y_1 = x_i$, $y_2 = x_j$, $y_3 = x_k$, and $y_4 = x$, where $\{x\} = \{x_1,\ldots,x_4\}\setminus \{x_i,x_j,x_k\}$. Then $M$ is one of the following ideals:
\item Ideals with exactly one minimal generator:
\[\overbracket{(y_1 y_2 y_3)}^{\#\textbf{47}}\]
\item Ideals with exactly two minimal generators:
\[\overbracket{(y_1y_2y_3,y_1y_2y_4)}^{\#\textbf{48}} \: \overbracket{(y_1y_2y_3,y_1y_3y_4)}^{\#\textbf{49}}\: \overbracket{(y_1y_2y_3,y_2y_3y_4)}^{\#\textbf{50}}\: \overbracket{(y_1y_2y_3,y_1y_4)}^{\#\textbf{51}}\: \overbracket{(y_1y_2y_3,y_2y_4)}^{\#\textbf{52}}\:
 \overbracket{(y_1y_2y_3,y_3y_4)}^{\#\textbf{53}}\:  \]
\[\overbracket{(y_1y_2y_3,y_4)}^{\#\textbf{54}}\]
\item Ideals with exactly three minimal generators:
\[ \overbracket{(y_1y_2y_3,y_1y_2y_4,y_1y_3y_4)}^{\#\textbf{55}}\:  \overbracket{(y_1y_2y_3,y_1y_2y_4,y_2y_3y_4)}^{\#\textbf{56}}\:  \overbracket{(y_1y_2y_3,y_1y_3y_4,y_2y_3y_4)}^{\#\textbf{57}}\:  \overbracket{(y_1y_2y_3,y_1y_2y_4,y_3y_4)}^{\#\textbf{58}}\]
\[ \overbracket{(y_1y_2y_3,y_1y_3y_4,y_2y_4)}^{\#\textbf{59}}\:  \overbracket{(y_1y_2y_3,y_2y_3y_4,y_1y_4)}^{\#\textbf{60}}\:  \overbracket{(y_1y_2y_3,y_1y_4,y_2y_4)}^{\#\textbf{61}}\:  \overbracket{(y_1y_2y_3,y_1y_4,y_3y_4)}^{\#\textbf{62}}\]
\[ \overbracket{(y_1y_2y_3,y_2y_4,y_3y_4)}^{\#\textbf{63}}\]
\item Ideals with exactly  four minimal generators:
\[ \overbracket{(y_1y_2y_3,y_1y_2y_4,y_1y_3y_4,y_2y_3y_4)}^{\#\textbf{64}}\: \overbracket{(y_1y_2y_3,y_1y_4,y_2y_4,y_3y_4)}^{\#\textbf{65}}\:\]
\end{itemize}
4) $\mathscr{M}_4: \: \overbracket{(y_1y_2y_3y_4)}^{\#\textbf{66}}\:$, (where $y_1 = x_1,\ldots,y_4 = x_4$).

This class of 66 ideals could be reduced to an even smaller family by relabeling the subscripts of the variables. For instance, by relabeling the variables, $\#41,\ldots,\#45$ could be regarded as the same ideal. However, 66 ideals is a number that we can handle and, paradoxically, making any further reductions could lead to complications.

\section{Characteristic independence in 4 variables}
Before giving a formula for the Betti numbers of a monomial ideal in 4 variables, we will show that Betti numbers in 4 variables are independent of the base field.

\begin{theorem}\label{Theorem 10}
The Betti numbers of monomial ideals in 4 variables are characteristic-independent.
\end{theorem}

\begin{proof}
Let $M$ be a monomial ideal of $S$. By Corollary \ref{Corollary 9}, $\betti_i(S/M) = \sum\limits_{m\in \mathbb{T}_M} \betti_{i,y_m}(T/M''_m)$. Thus, it is enough to prove that for each $m\in \mathbb{T}_M$, $\betti_{i,y_m}(T/M''_m)$ is characteristic-independent. Since $M''_m$ can be regarded as one of the 66 ideals listed in Section \ref{Section 2}, we will consider each case in particular. First, if $M''_m$ is a dominant ideal, minimally generated by at most three monomials then, according to Theorem \ref{Taylor}, $M''_m$ is minimally resolved by the Taylor resolution. Therefore, $\betti_{i,y_m}(T,M''_m)$ is characteristic-independent. On the other hand, if $M''_m$ is a nondominant ideal minimally generated by at most three monomials, then $M''_m$ must be minimally generated by exactly three monomials (if $M''_m$ were minimally generated by at most two monomials, $M''_m$ would be dominant), and its minimal resolution is obtained from its Taylor resolution by means of exactly one consecutive cancellation between one basis element in homological degree 2, and the only basis element in homological degree 3. Thus, $\betti_{i,y_m}(T/M''_m)$ is characteristic-independent. We have proven that the Betti numbers of the ideals $\#1,\ldots,\#28,\#47,\ldots,\#63,\#66$ in Section 4 are characteristic-independent.\\
Suppose now that $M''_m$ is minimally generated by four monomials. Then $M''_m$ can be regarded as one of $\#29,\ldots,\#40,\#64,\#65$. If $M''_m$ is one of $\#34,\#36$, then $M''_m$ is quadratic, and $T/M''_m$ is a Stanley-Reisner ring. Hence, $\betti_{i,y_m}(T/M''_m)$ is characteristic-independent [Ka, TH]. In addition, if $M''_m$ is one of $\#29,\ldots,\#33,\#35,\#37,\ldots,\#40$, then $M''_m$ has one dominant generator. Applying the third structural decomposition [Al1], we can express $\betti_{i,y_m}(T/M''_m)$ in terms of the Betti numbers of two monomial ideals in 3 variables, but such Betti numbers are characteristic-independent [MS, Exercise 3.1]. Notice that $\#64$ is a squarefree Veronese ideal, which implies that its Betti numbers are independent of the base field.\\
Finally, if $M''_m$ is of the form $\#65$, and we regard $k$ as a field of characteristic $0$, then the minimal resolution of $T/M''_m$ can be obtained from its Taylor resolution by making 3 consecutive cancellations. It can be verified that, each step, the entries of the differential matrices have coefficients 0, 1, or -1. Therefore, $\betti_{i,y_m}(T/M''_m)$ is characteristic-independent.\\
It only remains to prove that if $M''_m$ is one of $\#1,\ldots,\#66$, and it is minimally generated by more than 4 monomials, then $\betti_{i,y_m}(T/M''_m)$ is characteristic-independent; that is, we need to consider the case when $M''_m$ is one of $\#41,\ldots,\#46$. Note that, in each case, $M''_m$ is an edge ideal, and $T/M''_m$ is a Stanley-Reisner ring. By [Ka, TH], $\betti_{i,y_m}(T/M''_m)$ is characteristic-independent. 
\end{proof}

\section{Total Betti numbers}

We remind the reader of the notation that we have adopted. If $M$ is a monomial ideal in 4 variables, and $m$ is a multidegree of its Taylor resolution, then $M_m$ is the ideal generated by those minimal generators of $M$ that divide $m$. The ideals $M'_m$, and $M''_m$ are the twin and squarefree twin ideals of $M_m$, respectively. Finally, $G$, $G_m$, $G'_m$, and $G''_m$ are the minimal generating sets of $M$, $M_m$, $M'_m$, and $M''_m$, respectively.

The following table gives the multigraded Betti numbers $\betti_{2,y_m}(T/M''_m)$, and $\betti_{3,y_m}(T/M''_m)$, when $M''_m$ is any of the 66 squarefree ideals constructed in Section 4. Our computations can be easily verified using Macaulay2 [GS].

\captionof{table}{Second and third multigraded Betti numbers}\label{table}
\[
\begin{tabular}{| c | l | c |c|c| }
  \hline			
  \# & $ M''_m\quad\quad \quad \quad \quad \quad\quad \quad \quad \quad \quad\quad\quad \quad \quad \quad \quad\quad \quad \quad \quad \quad $& $y_m$ &$\betti_{2,y_m}  $ & $\betti_{3,y_m}$ \\
\hline
  1 & $(1)$& 1 &0&0 \\
\hline
  2 &$(y_1)$&$y_1$ &0&0 \\
  \hline  
3& $(y_1,y_2)$&$y_1y_2$&1&0\\
\hline
4&$(y_1,y_2,y_3)$&$y_1y_2y_3$&0&1\\
\hline
5& $(y_1,y_2,y_3,y_4)$&$y_1y_2y_3y_4$&0&0\\
\hline
6& $(y_1y_2)$&$y_1y_2$&0&0\\
\hline
7& $(y_1y_2,y_1y_3)$&$y_1y_2y_3$&1&0\\
\hline
8& $(y_1y_2,y_1y_4)$&$y_1y_2y_4$&1&0\\
\hline
9&$(y_1y_2,y_2y_3)$&$y_1y_2y_3$&1&0\\
\hline
10&$(y_1y_2,y_2y_4)$&$y_1y_2y_4$&1&0\\
\hline
11&$(y_1y_2,y_3y_4)$&$y_1y_2y_3y_4$&1&0\\
\hline
12&$(y_1y_2,y_3)$&$y_1y_2y_3$&1&0\\
\hline
13&$(y_1y_2,y_4)$&$y_1y_2y_4$&1&0\\
\hline
14&$(y_1y_2,y_1y_3,y_1y_4)$&$y_1y_2y_3y_4$&0&1\\
\hline
15&$(y_1y_2,y_1y_3,y_2y_3)$&$y_1y_2y_3$&2&0\\
\hline

16&$(y_1y_2,y_1y_3,y_2y_4)$&$y_1y_2y_3y_4$&0&0\\
\hline
17&$(y_1y_2,y_1y_3,y_3y_4)$&$y_1y_2y_3y_4$&0&0\\
\hline
18&$(y_1y_2,y_1y_4,y_2y_3)$&$y_1y_2y_3y_4$&0&0\\
\hline
19&$(y_1y_2,y_1y_4,y_2y_4)$&$y_1y_2y_4$&2&0\\
\hline
20&$(y_1y_2,y_1y_4,y_3y_4)$&$y_1y_2y_3y_4$&0&0\\
\hline
21&$(y_1y_2,y_2y_3,y_2y_4)$&$y_1y_2y_3y_4$&0&1\\
\hline
22&$(y_1y_2,y_2y_3,y_3y_4)$&$y_1y_2y_3y_4$&0&0\\
\hline
23&$(y_1y_2,y_2y_4,y_3y_4)$&$y_1y_2y_3y_4$&0&0\\
\hline
24&$(y_1y_2,y_1y_3,y_4)$&$y_1y_2y_3y_4$&0&1\\
\hline
25&$(y_1y_2,y_2y_3,y_4)$&$y_1y_2y_3y_4$&0&1\\
\hline
26&$(y_1y_2,y_1y_4,y_3)$&$y_1y_2y_3y_4$&0&1\\
\hline
27&$(y_1y_2,y_2y_4,y_3)$&$y_1y_2y_3y_4$&0&1\\
\hline
28&$(y_1y_2,y_3,y_4)$&$y_1y_2y_3y_4$&0&1\\
\hline
\end{tabular}\]

\[\begin{tabular}{| c | l | c |c|c| }
\hline
  \# & $ M''_m\quad\quad \quad \quad \quad \quad\quad \quad \quad \quad \quad\quad\quad \quad \quad \quad \quad\quad \quad \quad \quad \quad  $&$y_m$&$\betti_{2,y_m} $ & $\betti_{3,y_m}$ \\
\hline

29&$(y_1y_2,y_1y_3,y_1y_4,y_2y_3)$&$y_1y_2y_3y_4$&0&1\\
\hline
30&$(y_1y_2,y_1y_3,y_1y_4,y_2y_4)$&$y_1y_2y_3y_4$&0&1\\
\hline
31&$(y_1y_2,y_1y_3,y_1y_4,y_3y_4)$&$y_1y_2y_3y_4$&0&1\\
\hline
32&$(y_1y_2,y_1y_3,y_2y_3,y_2y_4)$&$y_1y_2y_3y_4$&0&1\\
\hline
33&$(y_1y_2,y_1y_3,y_2y_3,y_3y_4)$&$y_1y_2y_3y_4$&0&1\\
\hline
34&$(y_1y_2,y_1y_3,y_2y_4,y_3y_4)$&$y_1y_2y_3y_4$&0&1\\
\hline
35&$(y_1y_2,y_1y_4,y_2y_3,y_2y_4) $&$y_1y_2y_3y_4$&0&1\\
\hline
36&$(y_1y_2,y_1y_4,y_2y_3,y_3y_4)$&$y_1y_2y_3y_4$&0&1\\
\hline
37&$(y_1y_2,y_1y_4,y_2y_4,y_3y_4)$&$y_1y_2y_3y_4$&0&1\\
\hline
38&$(y_1y_2,y_2y_3,y_2y_4,y_3y_4)$&$y_1y_2y_3y_4$&0&1\\
\hline
39&$(y_1y_2,y_1y_3,y_2y_3,y_4)$&$y_1y_2y_3y_4$&0&2\\
\hline
40&$(y_1y_2,y_1y_4,y_2y_4,y_3)$&$y_1y_2y_3y_4$&0&2\\
\hline

41&$(y_1y_2,y_1y_3,y_1y_4,y_2y_3,y_2y_4)$&$y_1y_2y_3y_4$&0&2\\
\hline
42&$(y_1y_2,y_1y_3,y_1y_4,y_2y_3,y_3y_4)$&$y_1y_2y_3y_4$&0&2\\
\hline
43&$(y_1y_2,y_1y_3,y_1y_4,y_2y_4,y_3y_4)$&$y_1y_2y_3y_4$&0&2\\
\hline
44&$(y_1y_2,y_1y_3,y_2y_3,y_2y_4,y_3y_4)$&$y_1y_2y_3y_4$&0&2\\
\hline
45&$(y_1y_2,y_1y_4,y_2y_3,y_2y_4,y_3y_4)$&$y_1y_2y_3y_4$&0&2\\
\hline
46&$(y_1y_2,y_1y_3,y_1y_4,y_2y_3,y_2y_4,y_3y_4)$&$y_1y_2y_3y_4$&0&3\\
\hline
47&$(y_1 y_2 y_3)$&$y_1y_2y_3$&0&0\\
\hline
48&$(y_1y_2y_3,y_1y_2y_4)$&$y_1y_2y_3y_4$&1&0\\
\hline
49&$(y_1y_2y_3,y_1y_3y_4)$&$y_1y_2y_3y_4$&1&0\\
\hline
50&$(y_1y_2y_3,y_2y_3y_4)$&$y_1y_2y_3y_4$&1&0\\
\hline
51&$(y_1y_2y_3,y_1y_4)$&$y_1y_2y_3y_4$&1&0\\
\hline
52&$(y_1y_2y_3,y_2y_4)$&$y_1y_2y_3y_4$&1&0\\
\hline
53&$(y_1y_2y_3,y_3y_4)$&$y_1y_2y_3y_4$&1&0\\
\hline
54&$(y_1y_2y_3,y_4)$&$y_1y_2y_3y_4$&1&0\\
\hline
55&$(y_1y_2y_3,y_1y_2y_4,y_1y_3y_4)$&$y_1y_2y_3y_4$&2&0\\
\hline
56&$(y_1y_2y_3,y_1y_2y_4,y_2y_3y_4)$&$y_1y_2y_3y_4$&2&0\\
\hline
57&$(y_1y_2y_3,y_1y_3y_4,y_2y_3y_4)$&$y_1y_2y_3y_4$&2&0\\
\hline
58&$(y_1y_2y_3,y_1y_2y_4,y_3y_4)$&$y_1y_2y_3y_4$&2&0\\
\hline
59&$(y_1y_2y_3,y_1y_3y_4,y_2y_4)$&$y_1y_2y_3y_4$&2&0\\
\hline
\end{tabular}\]

\[\begin{tabular}{| c | l | c |c|c| }
\hline
  \# & $ M''_m\quad\quad \quad \quad \quad \quad\quad \quad \quad \quad \quad\quad\quad \quad \quad \quad \quad\quad \quad \quad \quad \quad  $&$y_m$&$\betti_{2,y_m} $ & $\betti_{3,y_m}$ \\
\hline
60&$(y_1y_2y_3,y_2y_3y_4,y_1y_4)$&$y_1y_2y_3y_4$&2&0\\
\hline
61&$(y_1y_2y_3,y_1y_4,y_2y_4)$&$y_1y_2y_3y_4$&1&0\\
\hline
62&$(y_1y_2y_3,y_1y_4,y_3y_4)$&$y_1y_2y_3y_4$&1&0\\
\hline
63&$(y_1y_2y_3,y_2y_4,y_3y_4)$&$y_1y_2y_3y_4$&1&0\\
\hline
64&$(y_1y_2y_3,y_1y_2y_4,y_1y_3y_4,y_2y_3y_4)$&$y_1y_2y_3y_4$&3&0\\
\hline
65&$(y_1y_2y_3,y_1y_4,y_2y_4,y_3y_4)$&$y_1y_2y_3y_4$&1&1\\
\hline
66&$(y_1y_2y_3y_4)$&$y_1y_2y_3y_4$&0&0\\
\hline
\end{tabular}\]

\begin{note}\label{despues decidimos}
Table \ref{table} gives the multigraded Betti numbers $\betti_{i,y_m}(S/M''_m)$, under the assumption that $y_m= \lcm G''_m$. We should note, however, that there are instances when $y_m$ and $\lcm G''_m$ do not agree. In Example \ref{Example 5}, for instance, $y_m=y_1 y_2 y_3 y_4 $ and $\lcm G''_m=y_1 y_2 y_3$. Since $\lcm G''_m$ always divides $y_m$, when these monomials are not equal we must have $\betti_{i,y_m}(T/M''_m)=0$. 
\end{note}

\begin{theorem} \label{Betti}
Let $M$ be a monomial ideal in $4$ variables. Then 
\begin{align*}
&\betti_2(S/M)= \# \{m\in \mathbb{T}_M: y_m=\lcm G''_m\text{, and } \#G''_m=2 \}+\\
& \# \{m \in \mathbb{T}_M: y_m=\lcm G''_m \text{, } \#G''_m=3 \text{, and }M''_m \text{ is } \text{2-semidominant}\}+\\
& \#\{m\in \mathbb{T}_M: y_m=\lcm G''_m\text{, and }G''_m\text{ consists of }1\text{ cubic and }3 \text{ quadratic monomials}\}+\\
& 2\, \#\{m\in \mathbb{T}_M: y_m=\lcm G''_m \text{, } \#G''_m=3 \text{, and }M''_m \text{ is }\text{3-semidominant}\}+\\ 
& 3\,\#\{m\in \mathbb{T}_M: y_m=\lcm G''_m\text{, and }G''_m\text{ consists of }4\text{ cubic monomials}\}.
\end{align*}
\end{theorem}

\begin{proof}
Let us denote by $\mathscr{N}$ the class of all $66$ squarefree ideals listed in Section \ref{Section 2}. According to Table \ref{table}, for every $m\in \mathbb{T}_M$, $\betti_{2,y_m} \in \{0,1,2,3\}$. If $\betti_{2,y_m} = 1$, then $y_m=\lcm G''_m$, and (after doing a change of variables) $M''_m$ is one of $\#3$, $\#7, \ldots, \#13$, $\#48,\ldots,\#54$, $\#61,\ldots,\#63$, $\#65$. Now,
$\{\#3, \#7, \ldots, \#13,\#48,\ldots,\#54\} = \{$ideals in $\mathscr{N}$ minimally generated by $2$ monomials$\}$;
$\{\#61,\ldots,\#63\} = \{2$-semidominant ideals in $\mathscr{N}$, minimally generated by 3 monomials$\}$; and 
$\#65$ is the only ideal in $\mathscr{N}$ minimally generated by $1$ cubic and $3$ quadratic monomials. Therefore, 
\begin{align*}
&\{m\in \mathbb{T}_M: \betti_{2,y_m} = 1\} = \{m \in \mathbb{T}_M: y_m=\lcm G''_m\text{, and } \#G''_m=2\} \cup\\
& \{m\in \mathbb{T}_M: y_m=\lcm G''_m \text{, } \#G''_m=3 \text{, and }M''_m \text{ is } \text{2-semidominant}\}\cup\\
& \{m\in\mathbb{T}_M: y_m=\lcm G''_m\text{, and }G''_m\text{ consists of 1 cubic and 3 quadratic monomials}\}.
\end{align*}
 If $\betti_{2,y_m} = 2$, then $ y_m=\lcm G''_m$, and (after doing a change of variables) $M''_m$ is one of $\#15,\#19,\#55,\ldots,\#60$.
Notice that $\{\#15,\#19,\#55,\ldots,\#60\} = \{3$-semidominant ideals of $\mathscr{N}$, minimally generated by 3 monomials$\}$. Therefore, 
\[\{m\in\mathbb{T}_M:\betti_{2,y_m}=2\} = \{m\in\mathbb{T}_M: y_m=\lcm G''_m \text{, } \#G''_m=3 \text{, and }M''_m \text{ is }\text{3-semidominant}\}.\]
 If $\betti_{2,y_m} = 3$, then $ y_m=\lcm G''_m$, and (after doing a change of variables) $M''_m$ is $\#64$, which is the only ideal of $\mathscr{N}$ minimally generated by 4 cubic monomials. Therefore, 
\[\{m\in\mathbb{T}_M:\betti_{2,y_m} = 3\} = \{m\in\mathbb{T}_M: y_m=\lcm G''_m\text{, and } G''_m\text{ consists of 4 cubic monomials}\}.\]
 Finally, by Corollary \ref{Corollary 9}, 
\begin{align*}
&\betti_{2}(S/M) = \sum\limits_{m\in\mathbb{T}_M} \betti_{2,y_m}(T/M''_m)= \\
&\sum\limits_{\{m\in\mathbb{T}_M:\betti_{2,y_m} = 1\}} \betti_{2,y_m} + \sum\limits_{\{m\in\mathbb{T}_M:\betti_{2,y_m} = 2\}} \betti_{2,y_m} + \sum\limits_{\{m\in\mathbb{T}_M:\betti_{2,y_m} = 3\}} \betti_{2,y_m}=\\
& \#\{m\in\mathbb{T}_M:\betti_{2,y_m}=1\} + 2\#\{m\in\mathbb{T}_M:\betti_{2,y_m} = 2\} + 3\#\{m\in\mathbb{T}_M:\betti_{2,y_m} = 3\}=\\
& \#\{m\in\mathbb{T}_M : y_m=\lcm G''_m \text{, and } \#G''_m=2 \} +\\
& \#\{m\in\mathbb{T}_M: y_m=\lcm G''_m \text{, } \#G''_m=3 \text{, and }M''_m \text{ is } \text{2-semidominant}\} +\\
&\#\{m\in\mathbb{T}_M: y_m=\lcm G''_m\text{, and }G''_m \text{ consists of 1 cubic and 3 quadratic monomials}\} +\\
& 2\#\{m\in\mathbb{T}_M: y_m=\lcm G''_m \text{, } \#G''_m=3 \text{, and }M''_m \text{ is }\text{3-semidominant}\}+\\
&3\#\{m\in\mathbb{T}_M: y_m=\lcm G''_m\text{, and }G''_m\text{ consists of 4 cubic monomials}\}.
\end{align*}
\end{proof}

\begin{note}\label{BettiNumbers}
We are now ready to compute all Betti numbers of an arbitrary monomial ideal $M$ in 4 variables. Indeed, if $M$ has minimal generating set $G$, then
\begin{align*}
&\betti_0(S/M)=1,\\
&\betti_1(S/M)= \#G,\\
&\betti_2(S/M) \text{ is given by Theorem \ref{Betti},} \\
&\betti_4(S/M) \text{ is given by Theorem \ref{fourthbetti}, and} \\
&\betti_3(S/M)=1+\betti_2(S/M)+\betti_4(S/M)-\#G \text{, by the characteristic of Euler-Poincaré}.
\end{align*}
\end{note}

\begin{example} \label{computations}
Let $M = (x_1^2 x_2^2, x_1^2 x_2 x_3, x_2 x_3 x_4^2,x_3^2 x_4^2)$. It is clear that $\betti_0(S/M) = 1$, and $\betti_1(S/M) = 4$. Since $M$ is nondominant, minimally generated by 4 monomials, it follows from Theorem \ref{fourthbetti}, that $\betti_4(S/M) = 0$. The critical step is the computation of $\betti_2(S/M)$, which we do next. By simple inspection, we can see that the 16 basis elements of $\mathbb{T}_M$ determine 11 different multidegrees (some basis elements have the same multidegree); namely, $m = 1, x_1^2x_2^2,x_1^2x_2x_3,x_2x_3x_4^2,x_3^2x_4^2,x_1^2x_2^2x_3,x_1^2x_2^2x_3x_4^2,x_1^2x_2^2x_3^2x_4^2,x_1^2x_2x_3x_4^2,x_1^2x_2x_3^2x_4^2,x_2x_3^2x_4^2$. Since the squarefree twin ideals $M''_m$, determined by $m = 1,x_1^2x_2^2,x_1^2x_2x_3,x_2x_3x_4^2$, and $x_3^2x_4^2$, are minimally generated by less than 2 monomials, they do not play any role in the computation of $\betti_2(S/M)$ (notice that the formula for $\betti_2(S/M)$ involves squarefree twin ideals for which $\#G''_m\geq 2$). In the next table we analyze the remaining 6 multidegrees.
\[
\begin {tabular}{|l|l|l|l|l|l|}
\hline
\#& $m$&$ y_m$ &$M_m$ &$ M'_m$&$M''_m$\\
\hline
$1$&$x_1^2x_2^2x_3$&$ y_1y_2y_3$ &$ (x_1^2x_2^2,x_1^2x_2x_3)$ &$( x_1^2x_2^2,x_1^2x_3)$ & $(y_1y_2,y_1y_3)$\\
\hline
$2$& $x_1^2x_2^2x_3x_4^2$&$y_1y_2y_3y_4$ & $(x_1^2x_2^2,x_1^2x_2x_3,x_2x_3x_4^2)$ & $(x_1^2x_2^2,x_1^2x_3,x_3x_4^2)$&$(y_1y_2,y_1y_3,y_3y_4)$\\
\hline
$3$&$x_1^2x_2^2x_3^2x_4^2$& $y_1y_2y_3y_4$&$(x_1^2x_2^2,x_1^2x_2x_3,x_2x_3x_4^2,x_3^2x_4^2)$&$(x_1^2,x_4^2)$&$(y_1,y_4)$\\
\hline
$4$&$x_1^2x_2x_3x_4^2$&$y_1y_2y_3y_4$&$(x_1^2x_2x_3,x_2x_3x_4^2)$&$(x_1^2x_2x_3,x_2x_3x_4^2)$&$(y_1y_2y_3,y_2y_3y_4)$\\
\hline
$5$&$x_1^2x_2x_3^2x_4^2$&$y_1y_2y_3y_4$&$(x_1^2x_2x_3,x_2x_3x_4^2,x_3^2x_4^2)$&$(x_1^2x_2,x_2x_4^2,x_3^2x_4^2)$&$(y_1y_2,y_2y_4,y_3y_4)$\\
\hline
$6$&$x_2x_3^2x_4^2$&$y_2y_3y_4$&$(x_2x_3x_4^2,x_3^2x_4^2)$&$(x_2x_4^2,x_3^2x_4^2)$&$(y_2y_4,y_3y_4)$\\
\hline
\end{tabular}\]
Let us consider case $\#3$. Since $y_m = y_1y_2y_3y_4\neq y_1y_4 = \lcm(G''_m)$, by  Theorem \ref{Betti}, $M''_m$ does not play any role in determining $\betti_2(S/M)$. Likewise, in cases $\#2$ and $\#5$, $M''_m$ is 1-semidominant, minimally generated by 3 monomials and, by Theorem \ref{Betti}, $M''_m$ does not contribute towards $\betti_2(S/M)$. Finally, in cases $\#1$, $\#4$, and $\#6$, $y_m = \lcm(G''_m)$, and $\#G''_m = 2$. According to Theorem \ref{Betti}, $\betti_2(S/M) = \#\{x_1^2x_2^2x_3,x_1^2x_2x_3x_4^2,x_2x_3^2x_4^2\} = 3$. Now, it follows from the characteristic of Euler-Poincaré that $\betti_3(S/M) = \betti_0(S/M)+\betti_2(S/M)+\betti_4(S/M)-\betti_1(S/M) = 1+3-4 = 0$. All in all, $\betti_0(S/M) = 1$, $\betti_1(S/M) = 4$, $\betti_2(S/M) = 3$, and $\betti_3(S/M) = \betti_4(S/M) = 0$.
\end{example}

\section{The advantage of having formulas}
In the previous section, we found an indirect way to compute $\betti_3(S/M)$. In this section, we give an explicit formula for $\betti_3(S/M)$. This formula will enable us to prove an interesting fact about an infinite family of ideals, something that we would not be able to do if we applied an algorithm to compute $\betti_3(S/M)$. Both formulas and algorithms are critical to the study of monomial resolutions, but they play complementary roles. This section underlines the advantage of having formulas. 

\begin{theorem} \label{ThirdBetti}
Let $M$ be a monomial ideal in $4$ variables. Then 
\begin{align*}
&\betti_3(S/M) = \#\{m\in \mathbb{T}_M:y_m = \lcm G''_m \text{, } \#G''_m=3 \text{, and }M''_m\text{ is dominant}\} +\\
&\#\{m\in\mathbb{T}_M:y_m = \lcm G''\text{, and }G''_m\text{ consists of 4 quadratic monomials}\}+\\
&\#\{m\in\mathbb{T}_M:y_m = \lcm G''\text{, and }G''_m\text{ consists of 1 cubic and 3 quadratic monomials}\}+\\
&2 \#\{m\in\mathbb{T}_M:y_m = \lcm G''\text{, and }G''_m\text{ consists of 1 linear and 3 quadratic monomials}\}+\\
&2\#\{m\in\mathbb{T}_M:y_m = \lcm G''\text{, and } \#G''_m=5\}+\\
&3\#\{m\in\mathbb{T}_M:y_m = \lcm G''\text{, and } \#G''_m=6\}.
\end{align*}
\end{theorem}

\begin{proof}
Let $\mathscr{N}$ denote the class of all $66$ squarefree ideals listed in section \ref{Section 2}. According to Table \ref{table}, for every $m\in\mathbb{T}_M$, $\betti_{3,y_m}\in \{0,1,2,3\}$. If $\betti_{3,y_m} = 1$, , then $y_m=\lcm G''_m$, and (after doing a change of variables) $M''_m$ is one of $\#4,\#14,\#21,\#24,\ldots,\#38,\#65$. Now,
$\{\#4, \#14, \#21,\#24,\ldots,\#28\} = \{$dominant ideals of $\mathscr{N}$, minimally generated by $3$ monomials$\}$;
$\{\#29,\ldots,\#38\} = \{$ideals of $\mathscr{N}$, minimally generated by 4 quadratic monomials$\}$; and 
$\#65$ is the only ideal in $\mathscr{N}$ minimally generated by $1$ cubic and $3$ quadratic monomials. Therefore,
\begin{align*}
&\{m\in \mathbb{T}_M: \betti_{3,y_m} = 1\} = \{m \in \mathbb{T}_M: y_m=\lcm G''_m \text{, } \#G''_m=3 \text{, and }M''_m\text{ is dominant}\} \cup\\
& \{m\in \mathbb{T}_M: y_m=\lcm G''_m \text{, and }G''_m\text{ consists of 4 quadratic monomials}\}\cup\\
& \{m\in\mathbb{T}_M: y_m=\lcm G''_m\text{, and }G''_m\text{ consists of 1 cubic and 3 quadratic monomials}\}.
\end{align*}
If $\betti_{3,y_m} = 2$, , then $y_m=\lcm G''_m$, and (after doing a change of variables) $M''_m$ is one of $\#39,\ldots,\#45$. Notice that $\{\#39,\#40\} = \{$ideals of $\mathscr{N}$, minimally generated by 1 linear and 3 quadratic monomials$\}$, and $\{\#41,\ldots,\#45\} = \{$ideals of $\mathscr{N}$, minimally generated by 5 monomials$\}$. Therefore,
\begin{align*}
&\{m\in \mathbb{T}_M: \betti_{3,y_m} = 2\} =\{m\in\mathbb{T}_M: y_m=\lcm G''_m\text{, and } \#G''_m=5\} \cup\\
&\{m \in \mathbb{T}_M: y_m=\lcm G''_m \text{, and }G''_m\text{ consists of 1 linear and 3 quadratic monomials}\}.
\end{align*}
If $\betti_{3,y_m} = 3$, , then $y_m=\lcm G''_m$, and (after doing a change of variables) $M''_m$ is $\#46$, which is the only ideal of $\mathscr{N}$ minimally generated by $6$ monomials. Therefore,
\[\{m\in\mathbb{T}_M:\betti_{3,y_m} = 3\} = \{m\in\mathbb{T}_M: y_m=\lcm G''_m\text{, and } G''_m\text{ consists of 6 monomials}\}.\]
Finally, by Corollary \ref{Corollary 9},
\begin{align*}
&\betti_{3}(S/M) = \sum\limits_{m\in\mathbb{T}_M}\betti_{3,y_m}(T/M''_m)=\\
&\sum\limits_{\{m\in\mathbb{T}_M:\betti_{3,y_m} = 1\}}\betti_{3,y_m}+\sum\limits_{\{m\in\mathbb{T}_M:\betti_{3,y_m} = 2\}}\betti_{3,y_m}+\sum\limits_{\{m\in\mathbb{T}_M:\betti_{3,y_m} = 3\}}\betti_{3,y_m}=\\
&\#\{m\in\mathbb{T}_M:\betti_{3,y_m} = 1\}+2\#\{m\in\mathbb{T}_M:\betti_{3,y_m} = 2\}+3\#\{m\in\mathbb{T}_M:\betti_{3,y_m} = 3\}=\\
& \#\{m\in \mathbb{T}_M:y_m = \lcm(G'') \text{, } \#G''_m=3 \text{, and }M''_m\text{ is dominant}\} +\\
&\#\{m\in\mathbb{T}_M:y_m = \lcm(G'')\text{, and }G''_m\text{ consists of 4 quadratic monomials}\}+\\
&\#\{m\in\mathbb{T}_M:y_m = \lcm(G'')\text{, and }G''_m\text{ consists of 1 cubic and 3 quadratic monomials}\}+\\
&2 \#\{m\in\mathbb{T}_M:y_m = \lcm(G'')\text{, and }G''_m\text{ consists of 1 linear and 3 quadratic monomials}\}+\\
&2\#\{m\in\mathbb{T}_M:y_m = \lcm(G'')\text{, and }G''_m\text{ consists of 5 monomials}\}+\\
&3\#\{m\in\mathbb{T}_M:y_m = \lcm(G'')\text{, and }G''_m\text{ consists of 6 monomials} \}.
\end{align*}  
\end{proof}

\begin{lemma}\label{Lemma 1}
Let $m$ be a multidegree of $\mathbb{T}_M$. Let $M_m= (m_1,\ldots,m_q)$, and denote by $M''_m = (m''_1,\ldots,m''_q)$ the squarefree twin ideal of $M_m$. If there are $1\leq i,j,k\leq q$ such that $m_k\mid \lcm(m_i,m_j)$, then $m''_k\mid\lcm(m''_i,m''_j)$.
\end{lemma}

\begin{proof}
Denote by $M'_m = (m'_1,\ldots,m'_q)$ the twin ideal of $M_m$. Let $m_s = x_1^{\alpha_{s_1}}\ldots x_4^{\alpha_{s_4}}$, for all $s=1,\ldots,q$. Then, $m=x_1^{\alpha_1}\ldots x_4^{\alpha_4}$, where $\alpha_r = \Max(\alpha_{1_r},\ldots,\alpha_{q_r})$. By definition, for all $s=1,\ldots,q$, $m'_s = x_1^{\beta_{s_1}}\ldots x_4^{\beta_{s_4}}$, where \\
$\beta_{s_r} = \begin{cases}
                       \alpha_r &\quad\text{if } \alpha_{s_r} = \alpha_r, \\
                         0         &\quad\text{otherwise.}
\end{cases}$\\
Therefore,
\begin{align*}
 \lcm(m_i,m_j) &= x_1^{\Max(\alpha_{i_1},\alpha_{j_1})}\ldots x_4^{\Max(\alpha_{i_4},\alpha_{j_4})}, \text{ and} \\
\lcm(m'_i,m'_j)& = x_1^{\Max(\beta_{i_1},\beta_{j_1})}\ldots x_4^{\Max(\beta_{i_4},\beta_{j_4})}.
\end{align*}
Since $m_k\mid\lcm(m_i,m_j)$, $\alpha_{k_r}\leq \Max(\alpha_{i_r},\alpha_{j_r})$, for all $r$. We will prove that $m'_k\mid\lcm(m'_i,m'_j)$ by showing that $\beta_{k_r}\leq \Max(\beta_{i_r},\beta_{j_r})$. If $\beta_{k_r} = 0$, then $\beta_{k_r}\leq \Max(\beta_{i_r},\beta_{j_r})$. On the other hand, if $\beta_{k_r} = \alpha_r$, we must have that $\alpha_{k_r} = \alpha_r$. Thus, $\Max(\alpha_{i_r},\alpha_{j_r}) = \alpha_r$, which means that either $\alpha_{i_r} = \alpha_r$ or $\alpha_{j_r} = \alpha_r$. Hence, either $\beta_{i_r} = \alpha_r$ or $\beta_{j_r} = \alpha_r$, and then $\Max(\beta_{i_r},\beta_{j_r}) = \alpha_r = \beta_{k_r}$. We have proven that $m'_k\mid \lcm(m'_i,m'_j)$. Finally, since $m''_k$, $m''_i$, $m''_j$ are just a different representation of $m'_k$, $m'_i$, $m'_j$, respectively, the result follows.
\end{proof}

\begin{theorem}\label{Theorem 2}
Let $M=(m_1, \ldots,m_q)$, with $q\geq 2$. Suppose that there is $1 \leq k \leq q$, such that $m_k\mid\lcm(m_i,m_j)$, for all $i \neq j$. Then $\pd(S/M) = 2$.
\end{theorem}

\begin{proof}
Let $m$ be a multidegree of $\mathbb{T}_M$. Let $M_m = (l_1,\ldots,l_r)$ be the ideal generated by the generators of $M$ that divide $m$. Denote by $M'_m = (l'_1,\ldots,l'_r)$ and $M''_m = (l''_1,\ldots,l''_r)$ the twin and squarefree twin ideals of $M_m$, respectively. We will show that $\betti_{3,m}(S/M) = 0$. By Theorem \ref{ThirdBetti}, it is enough to show that $M''_m$ is not a dominant ideal minimally generated by three monomials, or a quadratic ideal minimally generated by at least four monomials (according to Table \ref{table}, if $\#G''_m \geq 5$, then $M''_m$ must be quadratic), or an ideal minimally generated by one linear and three quadratic monomials, or an ideal minimally generated by one cubic and three quadratic monomials.

 If $r\leq 2$, then $M''_m$ cannot be one of the ideals above. Suppose then that $r\geq3$. By hypothesis, there is $1 \leq k \leq q$, such that $m_k\mid\lcm(m_i,m_j)$, for all $i \neq j$. In particular, $m_k\mid\lcm(l_1,l_2)$. Since $l_1,l_2\mid m$, it follows that $m_k\mid m$. Hence, $m_k$ is one of $l_1,\ldots,l_r$. Without loss of generality, we may assume that $m_k = l_r$. 

Suppose that $M''_m$ is a dominant ideal minimally generated by 3 monomials. Then $M''_m = (l''_1,\ldots,l''_r) = (l''_a,l''_b,l''_c)$, where $1\leq a,b,c\leq r$. Notice that $l''_r$ must be a multiple of one of $l''_a,l''_b,l''_c$; say $l''_r$ is a multiple of $l''_a$. By definition of dominance, $l''_a\nmid \lcm(l''_b,l''_c)$. It follows that $l''_r\nmid\lcm(l''_b,l''_c)$. By Lemma \ref{Lemma 1}, $l_r\nmid \lcm(l_b,l_c)$, which contradicts the hypothesis. We conclude that $M''_m$ is not a dominant ideal minimally generated by three monomials. 

Suppose now that $M''_m$ is a quadratic ideal, minimally generated by at least four monomials. That is, $M''_m = (l''_1,\ldots,l''_r) = (l''_a,l''_b,l''_c,l''_d,\ldots)$, where $1\leq a,b,c,d\leq r$ and $l''_a,l''_b,l''_c,l''_d$ are quadratic squarefree. Let $l''_a=x_ix_j$, and let $\{u,v\} = \{x_1,\ldots,x_4\}\setminus\{x_i,x_j\}$. Note that there are exactly three quadratic squarefree monomials divisible by $x_i$; they are $x_ix_j$; $x_iu$; $x_iv$. Since $l''_a$ is one of them, at most $2$ of $l''_b,l''_c,l''_c$ are divisible by $x_i$. If less than $2$ of $l''_b,l''_c,l''_d$ are divisible by $x_i$, then at least $2$ of $l''_b,l''_c,l''_d$ are not divisible by $x_i$. This implies that $l''_a$ does not divide the $\lcm$ of every pair of minimal generators of $M''_m$. On the other hand, if $2$ of $l''_b,l''_c,l''_d$ are divisible by $x_i$, then none of these $2$ is divisible by $x_j$. Once again, we conclude that $l''_a$ does not divide the $\lcm$ of every pair of minimal generators of $M''_m$. Since $l''_a$ is arbitrary, we conclude that no minimal generator of $M''_m$ divides the $\lcm$ of every pair of minimal generators of $M''_m$. By Lemma \ref{Lemma 1}, we conclude that no minimal generator of $M_m$ divides the $\lcm$ of every pair of minimal generators of $M_m$, a contradiction.

Next, suppose that $M''_m$ is minimally generated by one linear and three quadratic monomials. Then, $M''_m$ can be expressed in the form $M''_m = (l''_1,\ldots,l''_r) = (n''_1 = y_1y_2,n''_2 = y_1y_3,n''_3 = y_2y_3,n''_4 = y_4)$. For $1\leq i,j\leq 3$, $i\neq j$, $n''_i\nmid\lcm(n''_j,n''_4)$. Also,
 $n''_4 \nmid\lcm(n''_1,n''_2)$. Thus, no minimal generator of $M''_m$ divides the $\lcm$ of every pair of minimal generators of $M''_m$. By Lemma \ref{Lemma 1}, no minimal generator of $M_m$ divides the $\lcm$ of every pair of minimal generators of $M_m$, a contradiction. 

Finally, suppose that $M''_m$ is minimally generated by one cubic and three quadratic monomials. Then, $M''_m$ can be expressed in the form $M''_m = (l''_1,\ldots,l''_r) = (n''_1 = y_1y_2y_3,n''_2=y_1y_4,n''_3=y_2y_4,n''_4=y_3y_4)$. Note that $n''_1\nmid \lcm(n''_2,n''_3)$; likewise, $n''_2\nmid\lcm(n''_3,n''_4)$; $n''_3\nmid\lcm(n''_2,n''_4)$, and $n''_4\nmid\lcm(n''_2,n''_3)$. Thus, no minimal generator of $M''_m$ divides the $\lcm$ of every pair of minimal generators of $M''_m$. By Lemma \ref{Lemma 1}, no minimal generator of $M_m$ divides the $\lcm$ of every pair of minimal generators of $M_m$, a contradiction.

This proves that $\betti_{3,m} = 0$. Since $m$ is arbitrary, $\betti_3(S/M) = 0$. Hence, $\pd(S/M) = 2$. 
\end{proof}

\begin{example}
Let $M = (m_1 = x_1^2 x_2^2 x_3, m_2=x_1^2 x_2^2 x_4, m_3=x_1 x_3^2 x_4^2, m_4=x_2 x_3^2 x_4^2, m_5=x_1 x_2 x_3 x_4)$. It is easy to see that $m_5\mid\lcm(m_i,m_j)$, for all $i \neq j$. By Theorem \ref{Theorem 2},
$\pd(S/M) = 2$.
\end{example}

As Example \ref{computations} shows, the reciprocal to Theorem \ref{Theorem 2} does not hold.

\section{Final comments}
If an ideal $M$ of $S$ were chosen at random, and we had to guess which of $\betti_2=\betti_2(S/M)$ and $\betti_3=\betti_3(S/M)$ is larger, we would likely say that $\betti_2>\betti_3$. This impression may be due to the fact that most of the examples commonly studied involve ideals minimally generated by $3$, $4$ or $5$ monomials, and for such ideals, $\betti_2\geq \betti_3$. Indeed, suppose that $M=(m_1,\ldots,m_q)$, where $3\leq q\leq 5$. Then $\mathbb{T}_M$ has ${q\choose 2}$ and ${q\choose 3}$ basis elements in homological degrees $2$ and $3$, respectively. Let $\betti_2 = {q\choose 2} - k$. This means that, starting with $\mathbb{T}_M$, it is possible to perform $k$ consecutive cancellations between $k$ basis elements in homological degree $2$ and $k$ basis elements in homological degree $3$. Therefore, $\betti_3\leq {q\choose 3} - k\leq {q\choose 2} - k = \betti_2$. However, when $M$ is minimally generated by $q\geq 6$ monomials, ${q\choose 3}\geq {q\choose 2}$, and it is possible that $\betti_3$ will be larger than $\betti_2$. For instance, when 
$M =(x_1^3,x_1^2 x_2, x_1 x_2^2, x_2^3, x_3^3, x_3^2 x_4, x_3 x_4^2, x_4^3)$, we have $\betti_3= 24>22=\betti_2$. 

However, in order for $\betti_3$ to be larger than $\betti_2$ an additional condition must be satisfied. By the characteristic of Euler-Poincaré, $\betti_0-\betti_1+\betti_2-\betti_3+\betti_4 = 0$. Hence, 
$\betti_3 - \betti_2 = \betti_0 + \betti_4 - \betti_1 =1+\betti_4 - \betti_1$. Thus, in order for $\betti_3$ to be larger than $\betti_2$ it is necessary (but not sufficient) that $\pd(S/M) = 4$. How restrictive is this condition? De Loera, Hosten, Krone, and Silverstein proved that (in a probabilistic sense) almost all monomial ideals $M$ in $S$ satisfy $\pd(S/M) = 4$ [DHKS]. Thus, there is no obvious reason to believe that one of $\betti_2$ and $\betti_3$ has a better chance to be the larger number.

 In addition to this, the explicit formulas for $\betti_2$ and $\betti_3$ given by Theorems \ref{Betti} and \ref{ThirdBetti}, respectively, have some similarities that make it difficult to conjecture which Betti number is usually larger. We invite the reader to explore this problem.

On a different note, there are examples of monomial ideals in 6 variables whose Betti numbers are characteristic-dependent [Pe, Example 12.4], while the Betti numbers of monomial ideals in 4 variables are characteristic-independent, as we proved in Section 5. Therefore, the first case where we can encounter examples of ideals whose minimal resolutions depend on the base field must be in either 5 or 6 variables. Notice that the study of characteristic dependence in 5 variables can be carried out in the same manner as in 4 variables. In fact, the results of Section 3 where taken from [Al2], and hold for monomial ideals in n variables.  

  \bigskip

\noindent \textbf{Acknowledgements}: My family and I lived in the U.S. for many years. In 2018, because of problems with our visas, we had to leave the U.S. and move back to our home country, Argentina. At the time of this writing, we have lived in Argentina for several months, in the midst of exceptional hardship. In spite of our difficulties, God has provided for the needs of my family, and has surrounded me with the love and support of my wife Danisa. She encouraged me to keep exploring the wonderful world of mathematics and, when I finally found something worth mentioning, she patiently typed my work. Prov. 31:29.

\end{document}